\documentclass[12pt]{article}
\long\def\symbolfootnote[#1]#2{\begingroup\def\thefootnote{\fnsymbol{footnote}}
\footnote[#1]{#2}\endgroup}
\usepackage{amsmath,amssymb,amsthm,graphics}

\newtheorem{lemma}{Lemma}[section]

\newtheorem{prop}[lemma]{Proposition}

\newtheorem*{T1}{Theorem 1}
\newtheorem*{T2}{Theorem 2}
\newtheorem*{C2}{Corollary 2}
\newtheorem*{C3}{Corollary 3}

\theoremstyle{remark}
\newtheorem{remark}[lemma]{Remark}

\begin{document}
\title{The Reductive Subgroups of $G_2$}
\author{David I. Stewart}\date{}
\maketitle
{\small {\bf Abstract.} Let $G:=G_2(K)$ be a simple algebraic group of type $G_2$ defined over an algebraically closed field $K$ of characteristic $p>0$. Let $\sigma$ denote a standard Frobenius automorphism of $G$ such that $G_\sigma\cong G_2(q)$ with $q\geq 4$. In this paper we find all reductive subgroups of $G$ and quasi-simple subgroups of $G_\sigma$ in the defining characteristic. Our results extend the complete reducibility results of \cite[Thm 1]{ls1}.}

\section{Introduction}
Recall that $G_2$ has maximal rank subgroups of type $A_1\tilde A_1$ and $A_2$ (also $\tilde A_2$ generated by all short root groups of $G$ when $p=3$). When $p=2$ we define $Z_1$ to be the subgroup of type $A_1$ obtained from the embedding 
\[A_1(K)\to A_1(K)\circ A_1(K)\leq G ;\hspace{20pt} x\mapsto (x,x).\]
Also when $p=2$, we define $Z_2$ to be the subgroup of type $A_1$ obtained from the embedding
\[A_1(K) \to A_2(K)\leq G\]
where $A_1\cong PSL_2(K)$ embeds in $A_2$ by its action on the three-dimensional space Sym$^2 V$ for $V$ the standard module for $SL_2(K)$. It is shown later that these subgroups are contained in the long root parabolic of $G$, that is, $P=\langle B, x_{-r}(t):t\in K\rangle$ where $r$ is the long simple root associated with the choice of Borel subgroup $B$.

Let $\bar L$, (resp. $\tilde L$) denote the standard Levi subgroup of the standard long root (resp. short root) parabolic subgroup of $G$ containing the Borel subgroup $B$. Let $\bar L_0$ (resp $\tilde L_0$) denote the subgroup of $\bar L$ (resp $\tilde L$) generated by the unipotent elements. Observe that $\bar L_0\cong \tilde L_0\cong A_1$.

The main theorem is:

\begin{T1}Let $X\cong A_1(K)$ be a subgroup of a parabolic subgroup in $G=G_2(K)$.

If $p>2$ then $X$ is conjugate to precisely one of $\bar L_0$ and $\tilde L_0$.

If $p=2$ then $X$ is conjugate to precisely one of $\bar L_0$, $\tilde L_0$, $Z_1$ and $Z_2$.\end{T1}

Recall Serre's notion of $G$-complete reducibility \cite{Serre}. A subgroup is said to be $G$-completely reducible or $G$-cr if, whenever it is a subgroup of a parabolic subgroup of $G$, it is contained in a Levi subgroup of that parabolic subgroup.

\begin{C2} All connected reductive subgroups of $G$ are $G$-cr unless $p=2$, in which case there are precisely two classes of non $G$-cr subgroups.\end{C2}

This extends the result \cite[Thm 1]{ls1} which states that all subgroups of $G$ are $G$-cr provided $p>3$. 

\begin{C3} Let $X$ denote a closed, connected semisimple subgroup of $G$. Then up to conjugacy, $(X,p,V_7\downarrow X)$ is precisely one entry in the following table where $V_7\downarrow X$ denotes the restriction of the seven-dimensional Weyl module $W_G(\lambda_1)$ to $X$.
\begin{center}
\begin{tabular}{|c|c|c|}
\hline
$X$ & $p$ & $V_7\downarrow X$\\
\hline
$A_2$ & any & $10\oplus 01\oplus 0$\\
$\tilde A_2$ & $p=3$ & $11$\\
$A_1 \tilde A_1$ & any & $1\otimes \tilde 1\oplus 0 \otimes \tilde W(2)$\\
$\bar L_0$ & any & $1\oplus 1\oplus 0^3$ \\
$\tilde L_0$ & any &  $1\oplus 1\oplus W(2)$ \\
$Z_1$ & $p=2$ & $T(2)\oplus W(2)$  \\
$Z_2$ & $p=2$ & $W(2)\oplus W(2)^* \oplus 0$ \\
$A_1\hookrightarrow A_1 \tilde A_1; x\mapsto (x^{(p^r)},x^{(p^s)})$ $r\neq s$ & any & $(1^{(p^r)}\otimes1^{(p^s)})\oplus W(2)^{(p^s)}$ \\
$A_1\hookrightarrow A_2$, irred & $p>2$ & $2\oplus 2\oplus 0$\\
$A_1$, max & $p\geq 7$ & $6$\\
\hline
\end{tabular}
\end{center}
\end{C3}

The subgroup denoted $\tilde A_2$ exists only when $p=3$ and is generated by the short root subgroups of $G$. (The above table appears not to contain the irreducible $A_1\leq \tilde A_2$. It is shown later that this  subgroup is conjugate to the subgroup $A_1\hookrightarrow A_1\tilde A_1$ where $r=1$, $s=0$.)

Some remarks on notation. In the above table and elsewhere we refer to an irreducible module by its high weight $\lambda$. When $X$ is of type $A_1$, $\lambda$ is given as an integer; by a module $ab$ for a group of type $A_2$ we mean the irreducible module with high weight $a\lambda_1+b\lambda_2$ where $\lambda_i$ is the fundamental dominant weight corresponding to the simple root $\alpha_i$. By $V^{(p^r)}$ we mean the Frobenius twist of the module $V$ induced by the Frobenius morphism $x\mapsto x^{(p^r)}$. The notation $\mu_1|\mu_2|\dots|\mu_n$ indicates a module with the same composition factors as the module $\mu_1\oplus\mu_2\oplus\dots\oplus\mu_n$. The notation $\mu_1/\mu_2/\dots/\mu_n$ indicates an indecomposable module with composition factors of high weights $\mu_i$ for some dominant weights $\mu_i$ and is given in the order in which the factors occur so that there is a submodule $\mu_{i}/\dots/\mu_n$ and a quotient $\mu_1/\dots/\mu_{i-1}$.  By $W(2)$ we denote the Weyl module for $A_1$ of high weight $2$; when $p>2$ this is irreducible and when $p=2$ it is indecomposable of type $1^{(2)}/0$. Lastly when $p=2$ we denote by $T(2)$ the four-dimensional tilting module for $A_1$ which is indecomposable of type $0/1^{(2)}/0$. 

Now let $\sigma$ denote a standard Frobenius automorphism of $G$ such that $G_\sigma=G_2(q)$ with $q\geq 4$. We use the proof of Theorem 1 and its corollaries to prove a result about the quasi-simple subgroups of Lie type of $G_\sigma$ in the defining characteristic. (A quasi-simple group of Lie type is a perfect central extension of a simple group of Lie type.)

\begin{T2} Let $X(q_0)\leq G_\sigma$ where $X(q_0)$ is a quasi-simple group of Lie type over $\mathbb F_{q_0}$, a field of the same characteristic as $\mathbb F_q$. Then there exists a $\sigma$-stable simple algebraic subgroup $\bar X$ of $G$ of the same type as $X(q_0)$ containing $X(q_0)$.\end{T2}

\begin{remark} Using \cite[5.1]{ls4}, it follows that $X(q_0)$ is unique up to conjugacy in $\bar X_\sigma$. Since Corollary 3 determines $\bar X$, it follows that we have found, up to $G_\sigma$-conjugacy, all quasi-simple subgroups of Lie type of $G_\sigma$ with the same defining characteristic as $G$. \end{remark}

\begin{remark} The only non-simple semisimple subgroups of $G_\sigma$ are of the form $SL_2(q_1)\circ SL_2(q_2)$ with $q_1,q_2\geq 4$, since any such group must have rank $2$. Since we have found all the quasi-simple groups using the above theorem, we have also found all semisimple subgroups of Lie type of $G_\sigma$ in the defining characteristic. (A semisimple subgroup of Lie type, $H$ is a subgroup such that $H'=H$ and $H/Z(H)$ is a direct product of simple subgroups of Lie type.) 
\end{remark}

\section{Preliminaries}

Let $X\cong A_1(K)$ with $|K|\geq 4$ finite or $K$ algebraically closed of characteristic $p>0$. Let $V:=V_X(\lambda)$ denote an irreducible rational $KX$-module of high weight $\lambda$. To prove Theorem 1 we require some information about $H^1(X,V)$, the first cohomology group of $X$ with coefficients in $V$. We recall that $H^1(X,V)$ is a $K$-vector space and is in bijection with the $V$-conjugacy classes of closed complements to $V$ in the semidirect product $XV$. Recall also the standard fact that  $H^1(X,V)\cong Ext^1_X(K,V)$ (see \cite[p50]{jantz}).

\begin{lemma} $Ext^1_X(K,V_X(\lambda))$ is non-zero if and only if $\lambda$ is a Frobenius twist of the module $(p-2)\otimes 1^{(p)}$. When it is non-zero it is one-dimensional unless $|K|=9$ and $V_X(\lambda)=1\otimes 1^{(3)}$ where it is two-dimensional.
\end{lemma}
\begin{proof} This follows from setting $\mu=0$ in \cite[4.5]{ajl} with the small correction given in \cite[1.2]{ls3}.\end{proof}

Recall that a parabolic subgroup $P$ has a decomposition as a semidirect product $LQ$ of a Levi subgroup $L$ with unipotent radical $Q$. We employ the above result to investigate complements to $Q$ in $P$. The next result shows how $Q$ admits a filtration by $KL$-modules. We recall the notions of height, shape and level of a root from \cite{abs}. Take a root system $\Phi$ for $G(K)$ with fixed base of simple roots $\Pi$. Let $J\subset \Pi$ be a subset of the simple roots and define the parabolic subgroup $P_J$ by $P_J=\langle B, x_{-\alpha}(t): \alpha\in J\rangle$. Let $\Phi_J=\mathbb{Z}J\cap \Phi$. Fix a root $\beta\in\Phi^+ - \Phi_J$. We write $\beta=\beta_J+\beta_J'$ where $\beta_J=\sum_{\alpha_i\in J} c_i\alpha_i$ and $\beta_J'=\sum_{\alpha_i\in\Pi-J} d_i\alpha_i$.  Define
\begin{align*} \text{height}(\beta)&=\sum c_i+\sum d_i\\
\text{shape}(\beta)&=\beta_J'\\
\text{level}(\beta)&=\sum d_i.\end{align*}
Now define $Q(i):=\langle x_\beta(t): t\in K,\text{level}(\beta)\geq i\rangle$ and define $V_S=\langle x_\beta(t):t\in K, \text{shape}(\beta)=S\rangle$.

\begin{lemma}\label{abslem} Let $G(K)$ be a split Chevalley group. For each $i\geq1$, $Q(i)/Q(i+1)$ has the structure of a $KL$-module with decomposition $Q(i)/Q(i+1)=\prod V_S$, the product over all shapes $S$ of level $i$. Furthermore, each $V_S$ is a $KL$-module with highest weight $\beta$ where $\beta$ is the unique root of maximal height and shape $S$.\end{lemma}
\begin{proof} This is the main result of \cite{abs}, noting the Remark 1 at the end of the paper which gives the result even in the case $G(K)$ is special.\end{proof}

Throughout the paper we will need the restrictions $V_7\downarrow X$ of the seven-dimensional Weyl module $V_7:=W_{G_2}(\lambda_1)$ to various subgroups $X$ of $G=G_2(K)$. We calculate these now.

\begin{lemma} \label{table}The entries in the table following Corollary 3 have the restrictions $V_7\downarrow X$ as stated.\end{lemma}
\begin{proof} The restriction $V_7\downarrow X$ for the maximal $A_1$ when $p\geq 7$ is well known and is listed in \cite[Main Theorem]{Test}. 

Consider $G_2$ embedded in $D_4$ as the fixed points of the triality automorphism. We consider the restriction of the natural 8-dimensional module $V_8$ for $D_4$. Recall that $V_8\downarrow G_2=0/V_7$. For $p=2$, $V_7$ becomes reducible and $V_8\downarrow G_2=0/V_6/0$.

Recall that $\bar L_0, \tilde L_0$ are the simple, connected subgroups of the long and short Levi subgroups respectively. We first consider $V_7\downarrow \bar L_0, \tilde L_0$ and $A_1\tilde A_1$. 

We can see that $A_1\tilde{A_1}\leq A_1^4\leq D_4$. It is clear that the $A_1^4$ subsystem in $D_4$ is realised as $A_1\otimes A_1\perp A_1\otimes A_1\cong SO_4\perp SO_4$. Take the long $A_1$ to be the first of the four and the short $\tilde A_1$ to be embedded diagonally in the other three. 

Now it follows that we have $V_8\downarrow \tilde L_0=0^2\otimes 1\perp 1\otimes 1=1\oplus1\perp T(2)$ for $p=2$ and $1\oplus1\perp 2\oplus 0$ for $p>2$.  This gives $V_7\downarrow \tilde L_0=1\oplus 1\perp W(2)$ for $p=2$ and $V_7\downarrow \tilde L_0=1\oplus 1\perp 2$ for $p>2$. 

We also have $V_8\downarrow \bar L_0=1\otimes 0^2 \perp 0^2\otimes 0^2=1\oplus1\perp 0^4$. Hence $V_7\downarrow L_0=1\oplus 1\perp 0^3$. It follows also that $V_7\downarrow A_1\tilde A_1=1\otimes \tilde 1 \oplus 0 \otimes \tilde W(2)$.

Next we establish $V_7\downarrow A_2$. As the $A_2$ is a subsystem subgroup of $G_2$, it is in a subsystem of the $D_4$. It is therefore contained in an $A_3$. We can see easily that $\lambda_1$ for $D_4$ restricts to $A_3$ as $\lambda_1\oplus\lambda_3=\lambda_1\oplus\lambda_1^*$ (see e.g. \cite[13.3.4]{ca}). Since $A_2$ sits inside $A_3$ such that the natural module for $A_3$ restricts to $A_2$ as $\lambda_1\oplus 0$ we see that $V_7\downarrow A_2=\lambda_1 \oplus \lambda_1^*\oplus 0$. 

Using this we can restrict to the irreducible $A_1\leq A_2$ for $p>2$, and to $Z_2\leq A_2$, when $p=2$. In this case the natural module for $A_2$, $\lambda_1\downarrow A_1=2$ for $p>2$ and $\lambda_1\downarrow Z_2=W(2)$. Hence $V_7\downarrow A_1=2\oplus 2\oplus 0$ and $V_7\downarrow Z_2=W(2)\oplus W(2)^*\oplus 0$.

Now we compute $V_7\downarrow X$ for $X:=A_1\hookrightarrow A_1\tilde A_1$ twisted by $p^r$ on the first factor and $p^s$ on the second. Using the decomposition above, we read off $V_7\downarrow X=1^{(p^r)}\otimes 1^{(p^s)} \oplus 2^{(p^s)}$. For $s=r=0$ when $p=2$, this gives $V_7\downarrow Z_1=T(2)\oplus 2/0$. 

Lastly let $X=\tilde A_2$ ($p=3$). One checks that a base of simple roots $\{\beta_1,\beta_2\}$ for $G$ is expressed in terms of the roots of $D_4$ as $\{\frac{1}{3}(\alpha_1+\alpha_3+\alpha_4),\alpha_2\}$. On these two elements, the weight $\lambda_1$ for $D_4$ has $\lambda_1(\beta_1)=1$ and $\lambda_1(\beta_2)=1$ implying $V_8\downarrow \tilde A_2$ has composition factors $11|00$ so that $V_7\downarrow \tilde A_2=11$.
\end{proof}

\section{Complements in parabolics: proof of Theorem 1}

Let $G=G_2(K)$ with $K$ algebraically closed of characteristic $p$ and let $X\cong A_1(K)$ be a subgroup of $G$ contained in a parabolic subgroup $P=LQ$ of $G$. Then $X$ is a complement to $Q$ in $L_0 Q$, where $L_0$ denotes the simple subgroup of $L$ generated by the unipotent elements. In the cases we are considering $L_0=L'$. Recall the notation $\bar L_0$ and $\tilde L_0$ denoting the cases where $L_0$ is a long root $A_1$ and short root $A_1$ respectively.

\begin{lemma}\label{inlong} If $X$ is not conjugate to $L_0$, then $p=2$ and $X$ is contained in the long root parabolic subgroup of $G$.\end{lemma}
\begin{proof}
Using 2.2, for the short root parabolic one calculates that there are two levels in $Q$ and they have the structure of $KL_0$ modules with high weights $0$ and $3$ respectively. For $p>3$ they are restricted and thus irreducible. For $p=3$ they are the modules $0$ and $1^{(3)}/1$; for $p=2$ they are $0$ and $1^{(2)}\otimes 1$.

For the long parabolic one calculates that there are three levels with high weights $1$, $0$, and $1$ respectively. These are restricted and irreducible for all characteristics.

As $\bar L_0$ (resp. $\tilde L_0$) has some odd weights on the modules in $Q$, it is simply connected and hence admits a morphism $\phi$ to $X$. Composing this with the projection $\pi$ to the Levi factor, we have the morphism $\pi\circ\phi:L_0\to L_0$. It follows that $\pi\circ\phi$ is an isogeny. We may assume that this is the standard Frobenius morphism corresponding to $x\mapsto x^{(q)}$, say. This has the effect of twisting the modules found for $\bar L_0$ or $\tilde L_0$ above. Comparing these weights with 2.1, we see that none of the modules admitting a non-trivial $H^1$ is present unless $p=2$, $q$ is non-trivial, and $X$ is in the long parabolic, a complement to $Q$ in $\bar L_0 Q$.\end{proof}

From this point we assume that $p=2$, $X\leq P$ the long root parabolic, a complement to $Q$ in $\bar L_0 Q$ and $X$ is not conjugate to $\bar L_0$ . As $H^1(X,1^{(q)})$ is $1$-dimensional for all $q>1$ we may assume that $q=2$, observing that we can obtain any other complement  to $Q$ by applying a Frobenius map to an appropriate complement we get for $q=2$. 

Some notation is necessary for the next part of the paper. Recall the notation from \cite{ca} which uses $x_r(t)$ to refer to the root element with parameter $t$ corresponding to the root $r$. Since we are working entirely within $G$, we will use $x_i(t)$ for $i\in\{\pm 1,\dots\,\pm 6\}$. If we write $(a,b)$ for $a\alpha_1+b\alpha_2$ with $\alpha_1$ the short fundamental root and $\alpha_2$ the long fundamental root of $G$, then \[ [x_1,x_2,x_3,x_4,x_5,x_6]=[x_{(1,0)},x_{(0,1)},x_{(1,1)},x_{(2,1)},x_{(3,1)},x_{(3,2)}]\]  Under this notation and that of Lemma \ref{abslem} \begin{align*}Q=Q(1)&=\langle x_i(t) : i\in\{1,3,4,5,6\}\rangle\\ Q(2)&=\langle x_i(t) : i\in\{4,5,6\}\rangle\\ Q(3)&=\langle x_i(t) : i\in\{5,6\}\rangle.\end{align*} We see then that $Q/Q(2)$, $Q(2)/Q(3)$ and $Q(3)$ are modules for $X$ of high weights $2$, $0$ and $2$, respectively. 

\begin{lemma} Let $k,l\in K$. The groups $X_{k,l}$ generated by \begin{align*}x_+(t)&=x_2(t^2)x_3(kt)x_6(k^3t+lt){\textrm \ \ and }\\x_-(t)&=x_{-2}(t^2)x_1(kt)x_5(lt)\end{align*} for all $t\in K$ are closed complements to $Q$ in $\bar L_0 Q$. \end{lemma}
\begin{proof}

We certainly have $X_{k,l}Q=\bar L_0 Q$ as $\bar L_0 Q$ is generated by $\{x_i(t)\}$ for $i\in\{1,2,3,4,5,6,-2\}$. It remains to show that $X_{k,l}$ is isomorphic to $A_1(K)$, and it follows that $X_{k,l}\cap Q=\{1\}$ as required.

To show this we will check the generators and relations given in \cite[12.1.1 \& Rk. p198]{ca}, leaving us to show the following three statements hold:
\begin{enumerate}
\item $x_{\pm}(t_1)x_\pm(t_2)=x_\pm(t_1+t_2)$,
\item $h_+(t)h_+(u)=h_+(tu)$ and
\item $n_+(t)x_+(t_1)n_+(t)^{-1}=x_{-}(-t^{-2}t_1)$,
\end{enumerate}
for all $t_1,t_2\in K$ and $t,u\in K^\times$ where $n_+(t)=x_+(t)x_{-}(-t^{-1})x_+(t)$ and $h_+(t)=n_+(t)n_+(-1)$. We will abbreviate $n_{\alpha_i}(t)$ to $n_i(t)$, similarly for $h_{\alpha_i}(t)$.

Using the commutator relations for $G_2$ given in \cite[5.2.2]{ca} we show that these relations hold.

Write  \[\left[\begin{array}{c}i \\t\end{array}\right]:=x_i(t).\] 

Firstly, item (i) is easily checked: no positive linear combination of roots $\alpha_2$, $\alpha_3$ and $\alpha_6$ is a root except for the roots themselves, so 
\[\left[\begin{array}{c}2 \\t^2\end{array}\right], 
\left[\begin{array}{c}3 \\kt\end{array}\right] \text{ and }
\left[\begin{array}{c}6 \\k^3t+lt\end{array}\right],
\]
all commute with each other. The same argument follows for $x_{-}(t)$.

For (ii), we first calculate $n_+(t)$. So we must simplify
\[\left[\begin{array}{c}2 \\t^2\end{array}\right] 
\left[\begin{array}{c}3 \\kt\end{array}\right] 
\left[\begin{array}{c}6 \\k^3t+lt\end{array}\right]
\left[\begin{array}{c}-2\\t^{-2}\end{array}\right] 
\left[\begin{array}{c}1 \\kt^{-1}\end{array}\right]
\left[\begin{array}{c}5 \\lt^{-1}\end{array}\right]
\left[\begin{array}{c}2 \\t^2\end{array}\right] 
\left[\begin{array}{c}3 \\kt\end{array}\right] 
\left[\begin{array}{c}6 \\k^3t+lt\end{array}\right]\]
We will move all $\pm\alpha_2$ root elements to the left. The result of this calculation is
\[n_+(t)=\left[\begin{array}{c}2 \\t^2\end{array}\right] 
\left[\begin{array}{c}-2\\t^{-2}\end{array}\right]
\left[\begin{array}{c}2 \\t^2\end{array}\right] 
\left[\begin{array}{c}4 \\k^2\end{array}\right]=n_{2}(t^2)x_{4}(k^2).\]
Now it is easy to write down $h_+(t)$. Since $x_{\pm 2}(t)$ commute with $x_{4}(u)$ as $\alpha_4\pm\alpha_2$ are not roots we have 
\begin{align*}
h_+(t)&=\left[\begin{array}{c}2 \\t^2\end{array}\right] 
\left[\begin{array}{c}-2\\t^{-2}\end{array}\right]
\left[\begin{array}{c}2 \\t^2\end{array}\right] 
\left[\begin{array}{c}4 \\k\end{array}\right]
\left[\begin{array}{c}2 \\1\end{array}\right] 
\left[\begin{array}{c}-2\\1\end{array}\right]
\left[\begin{array}{c}2 \\1\end{array}\right] 
\left[\begin{array}{c}4 \\k\end{array}\right]\\
&=\left[\begin{array}{c}2 \\t^2\end{array}\right] 
\left[\begin{array}{c}-2\\t^{-2}\end{array}\right]
\left[\begin{array}{c}2 \\t^2\end{array}\right] 
\left[\begin{array}{c}2 \\1\end{array}\right] 
\left[\begin{array}{c}-2\\1\end{array}\right]
\left[\begin{array}{c}2 \\1\end{array}\right]\\
&=n_{2}(t^2)n_{2}(1)\\
&=h_{2}(t^2).
\end{align*}
It is then immediate that (ii) follows, since it holds for $h_{2}(t)$.
Part (iii) is similar.\end{proof}

Notice that $X_{0,0}=\bar L_0$ and so $X$ is not conjugate to $X_{0,0}$ by our standing assumption. The next two lemmas are necessary to show that the groups $X_{k,l}$ exhaust all closed complements to $Q$ in $\bar L_0 Q$.

\begin{lemma} The groups $X_{k,0}Q(2)$ are distinct up to $Q/Q(2)$-conjugacy in $XQ/Q(2)$ and so form a space isomorphic to $H^1(XQ(2)/Q(2),Q/Q(2))$.\end{lemma}

\begin{proof} $X_{k,0}Q(2)/Q(2)$ is generated by root groups $x_{+,k}(t)=x_2(t^2)x_3(kt)Q(2)$ and $x_{-,k}(t)=x_{-2}(t^2)x_1(kt)Q(2)$. Take a fixed, arbitrary element of $Q/Q(2)$, $g:=x_1(c_1)x_3(c_2)Q(2)$. Conjugating $x_{+,k}(t)$ by $g$ we get \[x_{+,k}(t)^g=x_2(t^2)x_3(c_1t^2+kt)Q(2)\] and accordingly for $x_{-,k}(t)^g$. Suppose these generate $X_{k',0}Q(2)/Q(2)$. Then we have an automorphism of $X_{k',0}Q(2)/Q(2)\cong PSL_2(K)$ extending the map $x_{+,k'}(t)\to x_{+,k}(t)\to x_{+,k}(t)^g$. This is an inner automorphism. So we must have both root groups $x_{+,k'}(t)$ and $x_{+,k}(t)^g$ conjugate, say
\[ x_{+,k'}(t)^{hQ(2)}=(x_2(t^2)x_3(k't)Q(2))^{hQ(2)}=x_2(t^2)x_3(c_1t^2+kt)Q(2)=x_{+,k}(t)^g\] for some $hQ(2)\in X_{k',0}Q(2)/Q(2)$. In particular they are conjugate modulo $Q$ in $X_{k',0}Q/Q$ by $hQ$. Then since $x_{+,k'}(t)Q=x_{+,k}(t)^gQ=x_2(t^2)Q$, $hQ$ must centralise $x_2(t^2)Q$ in $X_{k',0}Q/Q$. It follows that \[ \hspace{50pt} hQ=x_2(u_1)Q\hspace{50pt} (*)\] for some $u_1\in K$. 

Now, using the canonical form of \cite[8.4.4]{ca} any element $h$ of $X_{k,0}Q(2)/Q(2)$ is uniquely expressible as either  \begin{align*}h&=x_{+,k'}(v_1)h_2(v_2)Q(2) \text{ \ \ or}\\ h&=x_{+,k'}(v_1)h_2(v_2) n_2 x_{+,k'}(v_3)Q(2)\end{align*} where $n_2$ is a representative of the non-identity element of the Weyl group of $X_{k',0}Q(2)/Q(2)$. In the latter case, observe that modulo $Q$ we have $h=x_2(u_1^2)h_2(u_2)n_2x_2(u_3^2)Q$ which does not centralise $x_2(t^2)$ as it is not of the (unique) form (*) -- a contradiction. In the former case, observe that $v_2=1$ by (*) and so $hQ(2)$ centralises $x_{+,k'}(t)$.  So $c_1t^2+kt=k't$ for all $t\in K$. As there are at least four elements $t\in K$ this is impossible unless $c_1=0$ and $k=k'$. 

Lastly, to see that these complements form a space isomorphic to the space $H^1(XQ(2)/Q(2),Q/Q(2))$, observe that $X_{k,0}Q(2)$ is the closed complement corresponding to a rational cocycle $\gamma_k$, and we can define an addition $\gamma_k+\gamma_{k'}=\gamma_{k+k'}$ which is evidently well-defined on equivalence classes making the collection into a one-dimensional vector space as required.\end{proof}

\begin{lemma} The group $X_{k,l}$ is not conjugate to $X_{k,l'}$ by $Q(3)$ for $l\neq l'$. Thus for a fixed $k$, the groups $X_{k,l}$ form a space isomorphic to $H^1(X,Q(3))$.\end{lemma}
\begin{proof} The proof is similar to that of the previous lemma.
\end{proof}

\begin{prop} $X$ is $Q$-conjugate to $X_{k,l}$ for some $k,l\in K$, $k,l$ not both $0$.\end{prop}
\begin{proof} Firstly, observe that $XQ(2)/Q(2)$ must also be a complement to $Q/Q(2)$ in $X Q/Q(2)$. As $Q/Q(2)$ is a module for $X$ of high weight 2,  $H^1(X,Q/Q(2))=K$ and $X Q/Q(2)$ admits a one-dimensional collection of complements to $Q/Q(2)$. By 3.3 these are represented by $X_{k,0}Q(2)$. Replace $X$ by a $Q$-conjugate to have $XQ(2)=X_{k,0}Q(2)$.

Now observe $XQ(3)/Q(3)$ is a complement to $Q(2)/Q(3)$ in $X_{k,0}Q(2)/Q(3)$. As $Q(2)/Q(3)$ is a trivial module for $X$, we have $H^1(X,Q(2)/Q(3))=0$ and we may replace $X$ by a $Q$-conjugate to have  $XQ(3)=X_{k,0}Q(3)$.

Finally, observe that $X$ is a complement to $Q(3)$ in $X_{k,0}Q(3)$. As $Q(3)$ is a module for $X$ of high weight $2$, $H^1(X,Q(3))=K$ and $X_{k,0}Q(3)$ admits a one-dimensional collection of complements to $Q(3)$. By 3.4 these are represented by $X_{k,l}$. Thus we may replace $X$ by a $Q$-conjugate to have $X=X_{k,l}$.

Now, if $k=l=0$ then visibly $X_{k,l}\leq \bar L_0$ which we had earlier assumed was not the case.
 \end{proof}

\begin{lemma} The group $X_{k,l}$ is $P$-conjugate to one of $X_{1,0}$ or $X_{0,1}$.
\end{lemma}
\begin{proof} If $k\neq 0$, we can conjugate the generators of $X_{k,l}$ by the fixed element $x_{4}(l/k)$ by repeated use of Chevalley's commutator formula to get that \[x_{4}(l/k) X_{k,l} x_{4}(l/k)^{-1}=X_{k,0}.\] For instance, \begin{align*}  x_{4}(l/k)x_+(t)x_{4}(l/k)^{-1}&=x_{4}(l/k)x_2(t^2)x_3(kt)x_6(k^3t+lt)x_{4}(l/k)\\ &=x_2(t^2)x_3(kt)x_6(lt)x_6(k^3t+lt)\\&=x_2(t^2)x_3(kt)x_6(k^3t),\end{align*}
and the analogous calculation holds for the negative root group. Similarly we calculate that 
\[h_{4}(k)^{-1} X_{k,0} h_{4}(k)=X_{1,0}.\]
If $k=0$, and $l\neq 0$, again a calculation on the generators shows
\[h_{4}(c)^{-1} X_{0,l} h_{4}(c)=X_{0,1},\] where $c$ is any cube root of $l$.
\end{proof}

\begin{lemma} The groups $X_{1,0}$ and $X_{0,1}$ are conjugate to $Z_1$ and $Z_2$ respectively. The subgroups $Z_1$, $Z_2$ and $\bar L_0$ are pairwise non-conjugate in $G$. \end{lemma}
\begin{proof}
The construction of $Z_2$ as a subgroup of $A_2$ acting on the symmetric square representation allows us to calculate its root groups in terms of those of $A_2$. As the $A_2$ is a subsystem of $G$ it is easy to write these generators in terms of the root groups of $G$. Choosing the embeddings appropriately, one sees that $X_{0,1}$ has precisely the same generators, hence is conjugate to $Z_2$.

Next, for $p=2$, the module $V_7=W(\lambda_1)$ for $G$ is reducible and has a trivial submodule, so $V_7=V_6/0$ with $G\leq Sp(V_6)$. From the restriction $V_6\downarrow Z_1=W(2)\oplus W(2)^*$  in 2.3 we see that $Z_1$ stabilises a 1-space of $V_6$. Since the stabliser of a 1-space is parabolic, and $G$ acts transitively on all such by \cite[Thm B]{lss}, it follows that $Z_1$ is in a parabolic subgroup of $G$. Since it has a different restriction to $\tilde L_0$ it follows that from \ref{inlong} that it is in the long parabolic of $G$.

Now examine all the restrictions $V_7\downarrow Z_1, Z_2$ and $\bar L_0$ given by \ref{table}. One sees that they are all distinct. It follows that they are all distinct up to $G$-conjugacy. It now follows from \ref{inlong} that $X_{1,0}$ is in a long parabolic, not conjugate to $Z_2$ or $\bar L_0$ and so must be conjugate to $Z_1$.
\end{proof}

In conclusion we have established that a complement $X$ to $Q$ in $\bar L_0 Q$ must be conjugate to precisely one of the subgroups $Z_1, Z_2$ or $\bar L_0$. Together with $3.1$, this completes the proof of Theorem 1, and Corollary 2.

\section{Classification of semisimple subgroups of $G_2$: proof of Corollary 3}

In the proof of Corollary 3 we need the classification of maximal subgroups of the algebraic group $G=G_2(K)$, from \cite{ls2}.

\begin{lemma} Let M be a maximal closed connected subgroup of $G$. Then $M$ is one of the following:
\begin{enumerate}\item a maximal parabolic subgroup;\item a subsystem subgroup of maximal rank;\item $A_1$ with $p\geq 7$.\end{enumerate}\end{lemma}

{\it Proof of Corollary 3:}

Firstly, a semisimple subgroup in a parabolic of $G_2$ must be of type $A_1$ and we have determined these by Theorem 1. Secondly, the subsystem subgroups of $G_2$ are well known and can be determined using the algorithm of Borel-de Siebenthal. They are $A_2$, $A_1\tilde A_1$ and $\tilde A_2$ ($p=3$) where the $\tilde A_2$ is generated by the short roots of $G_2$.

Subgroups of maximal rank are unique up to conjugacy so to verify Corollary 3 it remains to check that we have listed all subgroups of type $A_1$ in subsystem subgroups in the table. If $X\cong A_1$ is a subgroup of $A_2$ or $\tilde A_2$ it must be irreducible or else it is in a parabolic; we have listed these in the table in Corollary 3. If $X\leq A_1\tilde A_1$, let the projection to the first (resp. second) factor be an isogeny induced by a Frobenius morphism $x\to x^{(p^r)}$ (resp. $x\to x^{(p^s)}$). We note some identifications amongst these subgroups:

When $p\neq 2$ and $r=s$ (without loss of generality $r=s=0$), $V_7\downarrow X=2\oplus 2\oplus 0$ which is the same as $V_7\downarrow Y$ where $Y:=A_1\hookrightarrow A_2$ where $Y$ acts irreducibly on the natural module for $A_2$. Indeed these are conjugate since $G$ acts transitively on non-singular 1-spaces (see \cite[Thm B]{lss}). When $p=2$ we get the subgroup $Z_1$. When $r=s+1$ and $p=3$, we have $V_7\downarrow X$ is a twist of $V_7 \downarrow Y$ where $Y$ is similarly irreducible in $\tilde A_2$, and we have actually $X$ conjugate to $Y$ up to twists: the long word in the Weyl group $w_0$ induces a graph automorphism on $\tilde A_2$ and it is easy to see that we can arrange the embedding $Y\leq \tilde A_2$ such that $Y\leq C_G (w_0)$. Now $C_G(w_0)=A_1\tilde A_1$ as there is only one class of involutions in $G$ when  $p\neq 2$ by \cite[p288]{iwahori}. The restriction $V_7\downarrow X,Y$ then gives the identification required.

Finally one can see that all other subgroups listed in the table of Corollary 3 are pairwise non-conjugate as the restrictions of $V_7$ in the table are all distinct.

This proves Corollary 3.

\section{Quasi-simple subgroups of $G_\sigma=G_2(q)$: proof of Theorem 2}

Let $X(q_0)$ be a finite quasi-simple subgroup of $G_\sigma=G(q)$, defined over a field of the same characteristic as $G$, where $q,q_0\geq 4$. We classify all such $X(q_0)$. For this we use the classification of maximal subgroups of $G_\sigma$. The following table is obtained from \cite[1.3A]{k} for $p>2$ and \cite{co} for $p=2$.

\begin{lemma} Let $M$ be a maximal subgroup of $G_\sigma=G_2(q)$ where $q=p^n \geq 4$. Then $M$ is conjugate to one of the following groups.
\begin{center}
\begin{tabular}{c c c c}
ID & Group & Structure & Remarks\\
\hline
(i) & $P_a$ & $[q^5]:GL_2(q)$ & parabolic\\
(ii) & $P_b$ & $[q^5]:GL_2(q)$ & parabolic\\
(iii) &$C_G (s_2)$ & $SL_2(q) \circ SL_2(q).(q-1,2)$ & involution centraliser \\
(iv) & $I$ & $2^3 . L_3 (2)$ & $q=p$, odd\\
(v) & $K_+$ & $SL_3(q):2$ & long\\
(vi) & $K'_+$ & $SL_3(q):2$ & $p=3$, short\\
(vii) & $K_-$ & $SU_3(q):2$ &  long\\
(viii) & $K'_-$ & $SU_3(q):2$ &  $p=3$, short\\
(ix) & $C_G(\phi)$ & $G_2(q_1)$ & $q=q_1^\alpha$, $\alpha$ a prime\\
(x) & $C_G(\phi)$ & $^2 G_2(q)$ & $p=3$, $n$ odd\\
(xi) & $PGL_2(q)$ & & $p\geq 7, q\geq 11$\\
(xii) & $L_2(8)$ & & only if $p\geq 5$\\
(xiii) & $L_2(13)$ & & $p\neq 13$, $GF(q)=GF(p)[\sqrt 13]$ \\ & & & or $q=4$\\
(xiv) & $G_2(2)$ & & $q=p\geq 5$\\
(xv) & $J_1$ & & $q=11$\\
(xvi) & $J_2$ & & $q=4$.
\end{tabular}
\end{center}
\end{lemma}

{\it Proof of Theorem 2:}

If $X(q_0)$ has rank 2 then it is $^2 G_2(q_0)$, $G_2(q_0)$ or $A_2(q_0)$ and one can see that  $X(q_0)\leq M$ where $M$ has ID (v)-(x) of the same type as $X(q_0)$: it is obvious for $X(q_0)$ of rank 2, $M$ cannot be as in cases (i)-(iv) and (xi); for cases (xii)-(xvi) one checks the appropriate pages in the Atlas \cite{atlas}. Such subgroups are unique up to $G_\sigma$-conjugacy by \cite[5.1]{ls4}. Therefore we have $X(q_0)\leq \bar X$ a $\sigma$-stable subgroup of $G$ of the same type.

We now consider the case where $X(q_0)$ has rank 1. Here $X(q_0)\cong A_1(q_0)$. We show that each of these is contained in a $\sigma$-stable connected subgroup of type $A_1\leq G$. Let $X(q_0)\leq M$, a maximal subgroup of $G_\sigma$. Firstly, if $M$ is case (i) or (ii), $X(q_0)\leq P_a$ or $P_b$ and we can use the proof of Theorem 1 to show that $X(q_0)$ is conjugate to a subgroup of a Levi or, when $p=2$, to a subgroup of one of the $\sigma$-stable subgroups $X_{k,l}\cong A_1$ defined above: 2.1 implies the groups $H^1(X,V)$ are still the same for all $q$ and $V$ being considered, 2.2 still applies for finite groups, and so 3.5 goes through to show that $X(q_0)\leq X_{k,l}$, a $\sigma$-stable subgroup of $G$ as required.

If $M$ is as in case (iii), $X(q_0)$ is embedded in $SL_2(q)\circ SL_2(q)$, twisted by $p^r$ on the first factor and $p^s$ on the second. We may assume $p^r,p^s<q$. Since $\sigma$ commutes with the twists on each factor, we have $X(q_0)\leq A_1$ where $A_1\hookrightarrow A_1\tilde A_1;x\mapsto (x^{(p^r)},x^{p^s})$ and is clearly $\sigma$-stable.

If $X(q_0)\leq M$ where $M$ has ID (iv) then $X(q_0)=L_2(7)\cong L_3(2)$. Checking \cite[p60]{atlas} one sees that the subgroup $2^3.L_3(2)$ is a non-split extension with normal subgroup $2^3$ so does not contain a subgroup of type $L_3(2)$.

We cannot have $X(q_0)\leq M$ if $M$ has ID (xii) or (xiii) as these do not contain subgroups of type $A_1(q_0)$, which is easily seen using \cite[p6 and p8]{atlas}

If $M$ has ID (xi), an $A_1(q_0)=L_2(q_0)$ in the $PGL_2(q)$ above is unique up to conjugacy and thus in the $\sigma$-stable maximal $A_1$. 

\begin{lemma} Let $M$ have ID (xiv) or (xv). Then $X(q_0)=L_2(7)$ or $L_2(11)$ resp. and it is conjugate to the subgroup $L_2(7)\leq PGL_2(7)$ or $L_2(11)\leq PGL_2(11)$ resp. with ID (xi) in the above list.\end{lemma}\begin{proof}
Pages 36 and 60 respectively of the Atlas substantiate the fact that we must have $X(q_0)=L_2(7)$ or $L_2(11)$ (rather than $L_2(7^2)$ or $L_2(11^2)$ for example). Examining the 7-dimensional Brauer characters in the Modular Atlas \cite{modat} of $L_2(7)\leq G_2(2)$ and $L_2(7)\leq PGL_2(7)$  one sees that they are irreducible and therefore conjugate in $GL_7(7)$. Similarly, the Brauer characters of $L_2(11)\leq J_1$ and $L_2(11)\leq PGL_2(11)$ in $G_2(11)$ are the same irreducible representation  and therefore conjugate in $GL_7(11)$. The result \cite[1.5.11]{k} then implies that they are conjugate in $G(q_0)$. Thus in each case, the subgroup $X(q_0)$ is in the $\sigma$-stable maximal $A_1$ of $G$.\end{proof}

\begin{lemma} Let $M$ have ID (v)-(viii). If $X(q_0)$ is a subgroup of $SL_3(q)$ or $SU_3(q)$ and is distinct from those already considered, then $q_0$ is odd and $X(q_0)$ is irreducible on the standard modules in each case. Moreover, each is contained in a  $\sigma$-stable subgroup of $\bar A_1\leq G$.\end{lemma}
\begin{proof} The action of $A_1(q_0)$ on the standard module $V$ for  $SL_3(q)$ or $SU_3(q)$ must be irreducible or else it is in a parabolic and already considered. It follows that $q_0$ is odd.

The fact that $A_1(q_0)$ is irreducible gives the restriction of the three-dimensional standard module as a high weight 2. Thus it is unique up to conjugacy in $GL_3(q)$ (or $GU_3(q)$) by \cite[2.10.4(iii)]{KL}. Hence it is contained in a $\sigma$-stable $\bar A_1\leq A_2$.\end{proof}

\begin{lemma} Let $M$ have ID (xvi). Then $X(q_0)=L_2(4)$ and $X(q_0)$ is contained in a subsystem subgroup and is thus already considered.\end{lemma}\begin{proof} Checking the maximal subgroups of $J_2$ in the Atlas, one establishes that $A_1(q_0)=L_2(4)\cong A_5$. A simple Magma \cite{ma} calculation in $G_2(4)$ shows all of these lie within subsystem subgroups.\end{proof}

Observe finally that if $X(q_0)\leq M$ for $M$ with ID (ix) or (x) then $X(q_0)$ is in $G_2(q_0)$ or $^2 G_2(q_0)$. It is thus in one of its maximal subgroups and has already been considered, completing the proof of Theorem 2 and this paper.

{\bf Acknowledgements.} This paper was prepared towards the author's PhD qualification under the supervision of Prof. M. W. Liebeck, with financial support from the EPSRC. We would like to thank Prof Liebeck for setting the problem and for his immeasurable assistance in its production. We would also like to acknowledge the anonymous referee for his or her helpful comments and corrections to this paper.

David Stewart, Department of Mathematics, Imperial College, London, SW7 2AZ, UK. \\ Email: dis20@cantab.net

\end{document}